\documentclass[12pt]{amsart}
\usepackage{graphicx}
\usepackage{amssymb}
\usepackage{amscd}
\usepackage[all]{xy}
\usepackage[T1]{fontenc}
\usepackage{relsize}
\usepackage{tgheros}
\usepackage{geometry}

\newcommand{\Z}{\mathbb{Z}}
\newcommand{\mP}{\mathbb{P}}
\newcommand{\Real}{\mathbb{R}}
\newcommand{\Nat}{\mathbb{N}}
\newcommand{\su}{\ensuremath{s\operatorname{-}u}}

\newtheorem{theorem}{Theorem}[section]
\newtheorem{proposition}{Proposition}[section]

\newtheorem*{thmA}{Theorem A}
\newtheorem*{thmB}{Theorem B}

\newcounter{exa}
\newtheorem{example}[exa]{Example}

\newtheorem{remark}{Remark}[section]

\theoremstyle{definition}
\newtheorem{definition}{Definition}[section]

\usepackage{hyperref}

\hypersetup{
    colorlinks = true,
    linkcolor = blue
}
\subjclass{37C40, 37D30, 60K37, 37H99}
\keywords{Random Walk in a Random Environment, Stationary Measures, Geodesic Flow}

\title[A random walk defined by a geodesic flow]{A Symmetric Random Walk Defined By The Time-One Map Of A Geodesic Flow}

\author{Pablo D. Carrasco}
\email{pdcarrasco@mat.ufmg.br}
\author{T\'ulio Vales}
\email{tuliovf@ufmg.br}
\address{ICEx-UFMG, Av.\@ Presidente Ant\^onio Carlos 6627, Belo Horizonte-MG, BR31270-901}

\date{\today}

\begin{document}

\begin{abstract}
    In this note we consider a symmetric random walk defined by a $(f,f^{-1})$ Kalikow type system, where $f$ is the time-one map of the geodesic flow corresponding to an hyperbolic manifold. We provide necessary and sufficient conditions for the existence of an stationary measure for the walk that is equivalent to the volume in the corresponding unit tangent bundle. Some dynamical consequences for the random walk are deduced in these cases. 
\end{abstract}

\maketitle

\section{Introduction and Main Theorems}\label{sec.Intro}

For the purposes of this work, a \emph{dynamical system} consists of a compact metric space $M$ together with a homeomorphism $f:M\to M$. The set of continuous functions on $M$ is denoted by $C(M)$, and we consider the uniform norm on it. The Borel $\sigma$-algebra on $M$ is denoted $\mathcal{B}_M$; all measures considered on $M$ are Borel probability measures, and we denote this set by $\mathcal{P}(M)$.

The problem that we are interested in is the following: given a continuous function $p:M\to (0,1)$, the map $f$ determines a random (Markov) process on $M$ where for a given point $x$ the probability of transition to $f(x)$ is $p(x)$, and the probability of transition to $f^{-1}(x)$ is $1-p(x)$. This way, each $x$ defines a random walk in $\Z$: denote by $S_n^x$ the position of a particle starting in the origin $0\in\Z$ at time $n\geq 0$, and let $X_n^x:\Z\to\{-1,1\}$ be the random variable that indicates whether the particle has jumped left or right at this time $n$. Then
\begin{align}\label{eq.probrwx}
&\mathrm{Prob}(X_{n+1}^x=1|S_n^x=k)=p(f^kx)\\
&\mathrm{Prob}(X_{n+1}^x=-1|S_n^x=k)=1-p(f^kx).
\end{align}
The resulting process is (a subclass of) what is called a \emph{Random Walk in Random Environment} (RWRE); choosing a point $x$ is considered as to fix some environment of the walk. We will summarize the parts of the theory needed in Section \ref{sec.RandomWalks}, and refer the interested reader to \cite{RWRE} for a thoroughgoing introduction to the topic.

For now we will state that the variation of the random walks with the environments is characterized by the Markov Operator $P_f:C(M)\to C(M)$,
\begin{equation}\label{def.P}
    P_f\psi(x)=p(x)\cdot \psi(fx)+(1-p(x))\cdot \psi(f^{-1}x).
\end{equation}
Since $P_f$ preserves continuous functions, it induces $P^{\ast}_f:\mathcal{P}(M)\to\mathcal{P}(M)$ with 
\begin{equation}\label{def.Past}
  \int \psi\cdot dP^{\ast}_f\mu=\int P_f\psi\cdot d\mu\quad \psi \in C(M), \mu \in \mathcal{P}(M).   
\end{equation}

If the environments are initially distributed with some probability $\mu$, the measure $(P^{\ast})^n\mu$ gives the distribution of the environments after $n$ units of time.

\smallskip

\begin{definition}
A measure $\nu\in \mathcal{P}(M)$ is $P$ - stationary if $P^{\ast}_f\nu=\nu$.
\end{definition} 

\smallskip

By a direct fix point argument one deduces the existence of at least one $P$ - stationary measure. Nevertheless, assuming that we start with a dynamically defined distribution $\mu$ (for example, if $\mu$ is $f$ - invariant) we are interested in the existence of a $P$ - stationary distribution that has some resemblance to $\mu$, in other words we want to address the following.

\smallskip

\noindent\textbf{Problem:} for an $f$ - invariant measure $\mu$ find a $P$ - stationary measure $\nu$ that is absolutely continuous/equivalent to $\mu$.

\smallskip

Questions of this type were first studied by Y. Sinai \cite{sinai1999simple} for irrational translations on tori, and put in a general framework by JP. Conze and Y. Guivarc \cite{Conze2000}. The case of Anosov diffeomorphisms is considered in \cite{simpleanosov} by V. Kaloshin and Y. Sinai. For a more recent article, and a up-to-date bibliography we direct the reader to the work of D. Dolgopyat and B. Fayad and M. Saprykina \cite{Dolgopyat2019}. We compare our results with the previous literature at the end of this part.

In this note we consider a different type of example that the ones already considered in the literature (either uniquely ergodic or completely hyperbolic), namely we will study the referred problem above when $f=f_1:M\to M$ is the time-one map of the geodesic flow corresponding to an hyperbolic manifold, whereas $\mu$ will be the Liouville (Lebesgue) measure. It seems that the available methods are not effective for studying the random walk defined by this kind of map, particularly in what is called the \emph{symmetric} case, i.e.\@ when
\[
    \int \log \frac{p}{1-p}\cdot d\mu=0.
\]
Here we use a geometrical approach to the problem. To state our result let us recall that in this case there exist a codimension-one distribution of the form $E^s\oplus E^u$ transverse to the flow direction. Although this distribution is non-integrable, both $E^s, E^u$ are and define what they are called the \emph{stable} and \emph{unstable} horocyclic foliations (in dimension three they are the well known horocycles of hyperbolic geometry). The leaves of these foliations are invariant by the action of the flow, and in particular by $f$. Moreover, $f$ exponentially contracts intrinsic distances for points in the same stable horocycle, while exponentially expands distances between points in the same unstable one. It is a well known fact (recalled in Section \ref{sec:PartialHyperbolic}) that the existence of these types of foliations persists by small $\mathcal{C}^1$ perturbations of $f$; for $g$ close to $f$  its contracting foliation $W^s_g$ will be called its \emph{stable foliation}, and its expanding foliation $W^u_g$ will be called its \emph{unstable foliation}.

\begin{definition}\label{def:suloop} Let $g:M\to M$ be a small $\mathcal{C}^1$ perturbation of $f$ so that its stable and unstable foliations $\mathcal{W}^s_g,\mathcal{W}^u_g$ are defined.
\begin{enumerate}
\item A set $\mathcal{C}=\{x_0,\ldots, x_{N-1},x_{N}=x_0\}$ is a \su\ loop (or periodic cycle) for $g$ if $x_i,x_{i+1}$ belong to the same leaf of $\mathcal{W}^s_g$ or $\mathcal{W}^u_g$, for all $i=0,\ldots, N-1$. 
\item If $\varphi:M\to \Real$ is H\"older we define $F(\mathcal{C})(\varphi)=\sum_{i=0}^{N-1}F(\mathcal{C};x_i\to x_{i+1})(\varphi)$  where
\[
F(\mathcal{C};x_i\to x_{i+1}):=\begin{cases}
\mathlarger{\sum_{n=0}^{+\infty}} \varphi(g^nx_i)-\varphi(g^nx_{i+1}) & x_{i+1}\in \mathcal{W}^s_g(x_i)\\
-\mathlarger{\sum_{n=-1}^{-\infty}} \varphi(g^nx_i)-\varphi(g^nx_{i+1}) & x_{i+1}\in \mathcal{W}^u_g(x_i).
\end{cases}.
\]
\end{enumerate}
\end{definition}

In the definition above, observe that since distances between points in the stable (unstable) foliations are contracted by $g$ (resp. $g^{-1}$) and $\varphi$ is H\"older, the series converge absolutely and $F(\mathcal{C})(\varphi)$ is well defined. The functionals $F(\mathcal{C})$ on the space of H\"older functions were introduced by A. Katok and A. Kononenko in \cite{Katok1996}, and provide an generalization of Livschitz' theory for hyperbolic systems. We are ready to state our main theorem.

\smallskip

\begin{thmA}
Consider $S$ a compact hyperbolic manifold, $M=T^1S$ its unit tangent bundle, $\mu$ the Liouville measure on $M$ and let $f:M\to M$ be the time-one map of the geodesic flow. Consider also $p:M\to (0,1)$ a H\"older continuous function such that
\[
    \int \log \varphi\cdot d\mu=0 \;,\quad \varphi=\frac{p}{1-p}.
\]
Then there exists $N$ a $\mathcal{C}^2$ neighborhood of $f$ such that for every $g\in N$ the following holds: the Random Walk on $M$ defined by $(g,p)$ has a $P$ - stationary measure equivalent to $\mu$ if and only if for every \su\ loop $\mathcal{C}$ it holds $F(\mathcal{C})(\log\varphi)=0$.

Moreover, the density of the $P$ - stationary measure is continuous. If furthermore $p$ is differentiable, then the $P$ - stationary measure is a smooth volume on $M$.
\end{thmA}

This theorem will be deduced from a more general result (Theorem B) that, in pursuit of conciseness for this part, we will state later.

\smallskip

Before moving further let us contextualize our results in terms of the available literature. As we mentioned, the study these type of problems was initiated by Sinai in \cite{sinai1999simple}; in this work he considers the Random Walk generated by an irrational rotation, in the symmetric and non-symmetric case. For the symmetric case he imposes some regularity conditions, both in the rotation angle (requiring it to be Diophantine), and a requirement on the speed of decay of the Fourier coefficients of the map $p$. In this context Sinai is able to prove the existence of a stationary measure $\nu$ equivalent to Lebesgue by solving a functional equation for the density of $\nu$, by using previous own work in RWRE (of probabilistic nature). As the reader can guess, the conditions on $p$ and the rotation angle are to deal with small denominator type problems for solving the functional equation.

Building on Sinai's original ideas, in \cite{simpleanosov} Kaloshin and Sinai study the case when the dynamics is given by a (transitive) Anosov diffeomorphism. They impose a stronger regularity condition on $p$ ($\mathcal{C}^2$) and show by a probabilistic argument that \emph{typically} in the data $(\mu,p)$ there is no corresponding stationary measure equivalent to $\mu$. In particular, if $\mu$ is a Gibbs measure and the process defined by $p$ satisfies what is called the Functional Central Limit Theorem, then there is no absolutely continuous stationary measure equivalent to $\mu$.

A different approach based in ergodic theory was given \cite{Conze2000}; here Conze and Guivarc'h also consider the case of the circle obtaining comparable results as Sinai, assuming that $\log p$ has bounded variation but without Diophantine requirements on the angle. The authors also consider the case of an Anosov diffeomorphism by reducing the problem to the case of a shift of finite type and using some standard thermodynamic formalism, but nonetheless, their characterization is in terms of the solvability of a general functional (cohomological) equation (cf. \ref{teo.existestationary}), which in general is difficult to establish. 

Our approach relies on the one of Conze and Guivarc'h, and provides a characterizations of the solvability of the aforementioned cohomological equation in terms of the geometrical data available in the system. The cited previous work rely on a rigid structure of the dynamics driving the Random Walk: either translations in compact Abelian groups (where Fourier analysis is available), or systems having a well understood symbolic representation as Anosov diffeomorphism for which, as a result, one has very precise statistical information. By tying the geometry of the dynamics to the problem we can consider systems where the previous methods do not apply, particularly it seems that our approach is well suited for studying walks driven by partially hyperbolic systems (see \ref{def.ph}). As an illustration, Theorem B (page \pageref{teo.B}) likely can be extended to other partially hyperbolic systems where the solution of cohomological equations over the driving dynamics is understood, such as ergodic automorphisms of the torus \cite{Veech1986}.

Additionally, introducing a further aspect to the study of the problem is always desirable, as it can shed some light even in previous results. For example, in the setting of Anosov diffeomorphisms, a generic type behavior of non-existence of stationary measures  equivalent to a given Gibbs one can be deduced from the discussion in the Appendix. Of course, the results of Kaloshin and Sinai are much more complete: here we are just trying to illustrate a particular application.       

\smallskip

The organization of the rest of the article is as follows. In Section \ref{sec.RandomWalks} we discuss the relevant probability background for understanding the proof of our result; economy in the presentation is sought, however we include a discussion comparing this with the ergodic theory approach, for convenience of the more dynamical inclined reader. Then in Section \ref{sec:PartialHyperbolic} we introduce the necessary geometrical tools that we will use, and in particular discuss some basic aspects of Partially Hyperbolic Systems. In Section \ref{sec:proof} we present the proof of our main result, and then in the next section we give some applications to the dynamics of the Random Walk. We also include an Appendix indicating how to adapt our results to the completely hyperbolic case. 
{}

\section{Random Walks determined by dynamical systems}\label{sec.RandomWalks}

We start by a formal description of the Random Walk on $M$. Let $\Sigma:=\{f,f^{-1}\}^{\Z_+}$ considered as the topological product of discrete spaces: it is a compact metrizable space. Extend $p$ to $\tilde{p}:\{f,f^{-1}\}\times M\to (0,1)$ with 
\begin{align*}
    &\tilde{p}(f,x)=p(x)\\
    &\tilde{p}(f^{-1},x)=1-p(x).
\end{align*}
By using for example the Hahn-Kolmogorov extension theorem, one establishes for each $x\in M$ the existence of a unique probability measure $\mP_x$ on $\Sigma$ satisfying for every $N\geq 1$, for every cylinder 
\[
[a_1,\cdots,a_{N}]=\{(f_n)_{n\geq 1}\in \Sigma:f_1=a_1,\cdots,f_{N}=a_{N}\}
\]
the equality
\[
\mP_x([a_1,\cdots,a_{N}])=\tilde{p}(a_1,x)\cdot \tilde{p}(a_2,a_1x)\cdots \tilde{p}(a_{N},a_{N-1}\cdots a_1x).
\]
This way, $(S_n^x)_{n\in \Nat}$ is a Markov Chain on $(\Sigma,\mP_x)$, that represents all possible random walks starting from the point $x$ with the corresponding probabilities \ref{eq.probrwx}. Observe that there exist natural passages among walks corresponding to different points (from the walk $\alpha$ starting at $x$ to the walk $\sigma(\alpha)$ starting at $\alpha_1(x)$, where $\sigma:\Sigma\to \Sigma$ is the shift map); from the dynamical system point of view is natural then to consider a skew product construction $F:X=\Sigma\times M\to X$,
\begin{equation}\label{eq.skew}
F(\alpha,x)=(\sigma(\alpha),\alpha_1(x))
\end{equation}
that encodes the aforementioned passages. This approach however is not very useful in the context that we are considering (where the transition probabilities depend on the point), so we will present an alternative construction and indicate its relation with the skew-product construction later.

Observe that we can use the collection $\{\mP_x\}_x$ to induce a process on $M$, by defining the Markov operator $P_f=P:C(M)\to C(M)$
\begin{align*}
P\phi(x)&=\int_{\Sigma} \phi\circ S_1^x d\mP_x=p(x)\phi(fx)+(1-p(x))\phi(f^{-1}x)\\
\nonumber &= \int_{\{f,f^{-1}\}} \phi(\alpha(x))\tilde{p}(d\alpha,x).
\end{align*}
By direct computation, for $n\geq 1$
\[
P^n\phi(x)=\int_{\Sigma} \phi\circ S_n^x\cdot d\mP_x.
\]
As mentioned in the introduction, after determining the operator $P$ we can act by duality on $\mathcal{P}(M)$. It is worth bringing to the attention of the reader that $P$ (as any Markov operator) defines a probability kernel\footnote{The use of the same letter for both the Markov operator and its associated probability kernel is common practice.} $P:M\times\mathcal{B}_M\to [0,1]$ with
\[
    P(x,A)=P(\chi_A)(x),
\]
where $\chi_A$ denotes the characteristic function of $A\in\mathcal{B}_{M}$.

\smallskip

Denote by $\Omega:=M^{\Nat}$ equipped with its product topology, and let $\mathcal{B}_{\Omega}$ be its Borel ($=$ product) $\sigma$-algebra. For $n\in \Nat$ we write $X_n:\Omega \to M$ the projection
\[
    X_n(\omega)=\omega_n,
\] 
and let $\mathcal{B}^{(n)}_{\Omega}$ be the sub-$\sigma$-algebra of $\mathcal{B}_{\Omega}$ generated by $\{X_0,\ldots, X_n\}$. If $A\in\mathcal{B}_{M}$ and $n\in \Nat$ we write
\[
    [A]_n:=\{\omega\in \Omega: \omega_n\in A\}=X_n^{-1}(A).
\]

Before going any further let us recall the concept of conditional expectation and disintegration of measures. 

\subsection{Conditional expectation}  Given a probability space $(\Omega,\mathcal{B}_{\Omega},\mathbb{Q})$ and $\mathcal{B}'\subset \mathcal{B}_{\Omega}$ a sub-$\sigma$-algebra, there exists a positive linear operator $\mathbb{E}(\cdot |\mathcal{B}'):L^1(\Omega,\mathcal{B}_{\Omega},\mathbb{Q})\to L^1(\Omega,\mathcal{B}',\mathbb{Q})$ such that for $\psi \in L^1(\mathcal{B}_{\Omega})$, $\mathbb{E}(\psi |\mathcal{B}')$ is characterized by:
\begin{enumerate}
\item $\mathbb{E}(\psi |\mathcal{B}')$ is $\mathcal{B}'$ measurable.
\item $ \int_{A} \mathbb{E}(\psi |\mathcal{B}')d\mathbb{Q}=\int_A \psi d\mathbb{Q}$ for all $A\in \mathcal{B}'$. 
\end{enumerate}
It follows that for every $p\geq 1, \|\mathbb{E}(\psi |\mathcal{B}')\|_p\leq \|\psi \|_p$. When $\psi=\chi_A, A\in\mathcal{B}_{M}$ we write
\[
    \mathbb{Q}(A|\mathcal{B}')=\mathbb{E}(\chi_A |\mathcal{B}')
\]  
and call $\mathbb{Q}(A|\mathcal{B}')$ the conditional measure of $A$ relative to $\mathcal{B}'$.

\smallskip

We have the following (see for example \cite{Neveu} chapter V).
 
\begin{theorem}\label{teo.canonicalchain}
    Given $\mu\in \mathcal{P}(M)$ there exists a unique probability $\mathbb{Q}_{\mu}$ on $\mathcal{B}_{\Omega}$ such that
    \begin{itemize}
    \item $(X_0)_{\ast}\mathbb{Q}_{\mu}=\mu$.
    \item $A\in \mathcal{B}_{\Omega}, n\in\Nat \Rightarrow \mathbb{Q}_{\mu}([A]_{n+1}|\mathcal{B}^{(n)}_{\Omega})(w)=P(\omega_n,A)$.
    \end{itemize}
\end{theorem}

\begin{example}\label{ex.deltax} 
In the case when $\mu=\delta_x, x\in M$ one obtains by induction that $\mathbb{Q}_{\delta_x}$ is supported on $\{\omega:\omega_0=x\}$ and 
\begin{align*}
    \mathbb{Q}_{\delta_x}&(\{\omega:\omega_0=x,\omega_1=\alpha_1(x),\omega_2=\alpha_2\alpha_1(x),\cdots, \omega_n=\alpha_n\cdots\alpha_1(x)\})\\
&=\mP_x([\alpha_1,\cdots,\alpha_n]).
\end{align*}
\end{example}    

\smallskip

Consider $T:\Omega\to \Omega$ the shift map. Using the uniqueness part of Theorem \ref{teo.canonicalchain} we obtain 
\[
    T_{\ast}\mathbb{Q}_{\mu}=\mathbb{Q}_{P^{\ast}\mu},
\]
and thus for every $P$ - stationary measure $\nu$ on $M$ we get a dynamical system $T:(\Omega,\mathbb{Q}_{\nu})\to (\Omega,\mathbb{Q}_{\nu})$.

\smallskip

\begin{definition}
The dynamical system $T:(\Omega,\mathbb{Q}_{\nu})\to (\Omega,\mathbb{Q}_{\nu})$ will be referred as the dynamical system associated to the random walk (defined by $P$ and the stationary measure $\nu$).
\end{definition}

Let us elucidate the relation between $T$ and the skew-product construction \eqref{eq.skew}. Define $\Phi:X\to \Omega$,
\begin{equation}\label{eq.PHI}
    \Phi(\alpha,x)=(x,\alpha_1(x),\alpha_2\alpha_1(x),\cdots).
\end{equation}
For $A_0,\cdots, A_n\in \mathcal{B}_M$ note that
\begin{align*}
&\Phi^{-1}(A_0\times A_1\times\cdots\times A_n\times M\times M\times\cdots)=\bigcap_{k=0}^n F^{-k}(\Sigma\times A_k)\\
&=\{(\alpha,x):x\in A_0, \alpha_1(x)\in A_1,\cdots, \alpha_n\alpha_{n-1}\cdots\alpha_1(x)\in A_n\}\in \mathcal{B}_X.
\end{align*}
Note that $\Phi$ is invertible, except when $\omega$ contains a point $x\in M$ satisfying $f^2x=x$. The set of periodic orbits for (perturbations of) the geodesic flow has zero Lebesgue measure, so if the stationary measure $\nu$ is equivalent to Lebesgue then we can ignore these periodic points. We define $\mathcal{B}_C:=\Phi^{-1}(\mathcal{B}_{\Omega})$, and use $\Phi$ to induce measures $m_C, m_x=\mP_x\in\mathcal{P}(X)$ such that $\Phi_{\ast}m_C=\mathbb{Q}_{\nu}, \Phi_{\ast}m_x=\mathbb{Q}_{\delta_x}$; in particular, the maps
\begin{align*}
\Phi: (X,\mathcal{B}_C,m_C)\to (\Omega,\mathcal{B}_{\Omega}, \mathbb{Q}_{\nu})\\
\Phi: (X,\mathcal{B}_C,\mP_x)\to (\Omega,\mathcal{B}_{\Omega}, \mathbb{Q}_{\delta_x})
\end{align*}
are measure-theoretic isomorphisms. 

Finally, note that $F:(X,\mathcal{B}_C)\to (X,\mathcal{B}_X)$ is measurable, and since $\Phi\circ F=T\circ \Phi$, we have that 
$F_{\ast}m_c=m_c$ and:

\begin{proposition}\label{pro.isomorskewchain} 
The map $\Phi: (X,\mathcal{B}_C,m_C)\to (\Omega,\mathcal{B}_{\Omega}, \mathbb{Q}_{\nu})$ is an conjugacy between $(F,m_c)$ and $(T,\mathbb{Q}_{\nu})$, in the sense that
\begin{enumerate}
\item $\Phi$ measure-theoretic isomorphism.
\item $\Phi\circ F=T\circ \Phi$.
\end{enumerate}
\end{proposition}   

The Proposition above tell us we can use the skew-product \eqref{eq.skew} to study the dynamics of the random walks defined by $P$, but we have to use a different $\sigma$-algebra on $X$. Now we want to understand the relation between the measures $m_x$ and $m_C$. Let us recall the following construction.

\subsection{Disintegration of measures} Consider a Lebesgue probability space $(\Omega,\mathcal{B}_{\Omega},\mu)$;  that is, measure theoretic isomorphic to the unit interval equipped with its Lebesgue $\sigma$-algebra and a Lebesgue-Stieltjes measure. Let $\Omega/\mathcal{H}$ be the set of atoms, and $\pi: \Omega \to \Omega/\mathcal{H}$ the map that assigns to each $x$ the unique atom where is contained; $\pi$ is well defined for $\mu$ - almost every point. We equip $\Omega/\mathcal{H}$ with $\hat{\mathcal{B}}_{\Omega}$  the largest $\sigma$ - algebra making $\pi: (\Omega,\mathcal{B}_{\Omega})\to (\Omega/\mathcal{H},\hat{\mathcal{B}}_{\Omega})$ measurable, and let $\hat{\mu}=\pi_{\ast}\mu$.

\begin{definition}
 The partition $\mathcal{H}$ is said to be a measurable partition if $(\Omega/\mathcal{H},\hat{\mathcal{B}}_{\Omega},\hat{\mu})$
 is a Lebesgue space.
\end{definition}  

See \cite{rokhlin1967lectures} for more complete discussion of these topics, in particular for the proof of the following result.

 \begin{theorem}[Rohklin]
 Let $\mathcal{H}$ be a measurable partition of $(\Omega,\mathcal{B}_{\Omega},\mu)$. Then there exists a disintegration of $\mu$ relative to $\mathcal{H}$, i.e.\@ there exists a family $\{\mu^H\}_{H\in \Omega/\mathcal{H}}$, where each $\mu^H$ is a probability measure on $H$ satisfying: for any $A\in\mathcal{B}_{\Omega}$ the set $A\cap H$ is measurable with respect to the $\sigma$-algebra generated by $\mathcal{H}$, for $\hat{\mu}$-almost every $H\in\Omega/\mathcal{H}$. Moreover the function $H\to \mu^H(A\cap H)$ is measurable and
 \[
 \mu(A)=\int_{\Omega/\mathcal{H}} \mu^H(A\cap H)d\hat{\mu}(H).
 \]
\end{theorem}

It follows directly from the properties that disintegrations are essentially unique: if $\{\mu^H\}_{H\in \Omega/\mathcal{H}}, \{\tilde{\mu}^H\}_{H\in \Omega/\mathcal{H}}$ are disintegrations of $\mu$ relative to $\mathcal{H}$ then $\mu^H=\tilde{\mu}^{H}$ for $\hat{\mu}$ - almost every $H$.

\smallskip

In our case we have the following.

\begin{proposition}\label{pro.desintegra}
Let $\mathcal{H}=\left\{\Sigma\times\{x\}\right\}_{x\in M}$. Then $\{\mP_x\}_{x\in M}$ is the disintegration of $m_C$ in the partition $\mathcal{H}$, and the quotient measure on $X/\mathcal{H}\approx M$ is $\nu$.
\end{proposition}

\begin{proof}
It suffices to observe that the disintegration of $\mathbb{Q}_{\nu}$ by the partition $\{\{\omega: \omega_0=x\}\}_{x\in M}$ is given by $\{\mathbb{Q}_{\delta_x}\}_{x\in M}$, with quotient measure $\nu$. This follows by Theorem \ref{teo.canonicalchain}, as it leads to $\mathbb{Q}_{\nu}=\int \mathbb{Q}_{\delta_x} d\nu(x)$, which is easily seen to imply the claim.
\end{proof}

\smallskip

\begin{example}
Assume that $p(x)\equiv p$ constant. Then $\mathbb{P}_x=\mP$ does not depend on $x$, and is given by the Bernoulli (product) measure $(p,1-p)^{\Z_+}$ on $\Sigma$. This implies that $m_C=\mP\times \nu$ is the product measure, and thus extends to the whole $\sigma$-algebra $\mathcal{B}_X$. The extension is $F$-invariant.

The resulting system is called a (locally constant) random dynamical system, and is by far much more studied than its $T$ counterpart. See for example \cite{Kifer}.
\end{example}

\subsection{Absolutely continuous stationary measures}\label{subsec:abscont}

We now focus our attention to the problem of the existence of $P$ - stationary measures with additional dynamical properties. Let us observe that by the form of the operator $P$  one readily verifies that if $\nu$ is $P$ - invariant, then in particular is $f$- quasi-invariant, namely $f_{\ast}\nu$ is absolutely continuous with respect to $\nu$. Now given a quasi-invariant measure $\nu$, $f_{\ast}^{-1}\nu=h\nu$, if $\nu$ is equivalent to $\mu$ (the $f$-invariant measure) with density $\rho$, then
\begin{align*}
&\nu=\rho\cdot \mu\Rightarrow f_{\ast}^{-1}\nu=f_{\ast}^{-1}(\rho\cdot \mu)=\rho\circ f\cdot \mu=h\nu=h\rho\mu\\
&\Rightarrow h=\frac{\rho\circ f}{\rho}
\end{align*}
and in particular $\log h=u\circ f-u$ is a coboundary for $f$. It is (part of) a result due to J.P. Conze and Y. Guivar'c that in the symmetric case, the fact of $\log \frac{p}{1-p}$ being a coboundary is equivalent to the existence of a $P$ - stationary measure $\nu$ equivalent to $\mu$ (assuming ergodicity of the later measure).

For the rest of this work we fix 
\begin{itemize}
    \item $\mu$ an $f$ - invariant ergodic measure.
    \item $p:M\to (0,1)$ continuous satisfying the following symmetry condition
    \[
    \int \log \varphi(x) d\mu(x)=0,
    \]
    where $\varphi(x):=\frac{p(x)}{q(x)}, q(x):=1-p(x)$.
\end{itemize}

The following Theorem is proved in \cite{Conze2000}.

\begin{theorem}\label{teo.existestationary}
In  the hypotheses above, there exists a $P$ - stationary measure $\nu$ equivalent to $\mu$ if and only if $\log\varphi=\log \frac{p}{1-p}$ is a (integrable) coboundary over $f$, i.e. there exists $\phi \in L^1(\mu), \phi>0$ such that
\[
\log\varphi=\phi\circ f -\phi. 
\]
\end{theorem}

In \cite{Conze2000} the result is stated changing `$\log\varphi=\log \frac{p}{1-p}$ is coboundary over $f$' by `there exists $\psi$ measurable such that $\frac{p}{q\circ f}=\frac{\psi\circ f}{\psi}$', but it is easy to see that these are equivalent in our case. Indeed, if 
\begin{equation}\label{eq.densidadCG}
\frac{p}{q\circ f}=\frac{\pi\circ f}{\pi}
\end{equation}
then 
\[
\frac{p}{q}=\frac{(\pi\cdot q)\circ f}{\pi\cdot q},
\]
whereas if 
\begin{equation}
\frac{p}{q}=\frac{\psi\circ f}{\psi}
\end{equation}
then 
\[
\frac{p}{q\circ f}=\frac{(\frac{\psi}{q})\circ f}{\frac{\psi}{q}}.
\]
It is also worth pointing out that if \eqref{eq.densidadCG} holds, then $\nu=\pi\mu$ is $P$ - invariant.

Concerning the uniqueness of the $P$ - stationary measure $\nu$ equivalent to $\mu$, it is direct consequence of Hopf's ergodic theorem. See Proposition $2.3$ in \cite{Conze2000} for details.

\section{Partial Hyperbolicity and Invariant Foliations}\label{sec:PartialHyperbolic}

As in the introduction, consider an hyperbolic closed manifold $S$ and let 
$M=T^1S$, $f:M\to M$ the time-one map of the geodesic flow. This map is an example of a \emph{partially hyperbolic diffeomorphism}, whose definition is recalled below.

\smallskip 

\noindent\textbf{Convention:} all manifolds considered are connected and second countable. By a submanifold we mean an immersed submanifold. If $M, N$ are submanifolds then $\mathrm{Emb}^r(M,N)$ denotes the set of embeddings of differentiability class $\mathcal{C}^r$ from $M$ to $N$.
\smallskip

\begin{definition}\label{def.ph}
 A diffeomorphism $f: M \to M$ of a compact manifold is partially hyperbolic if there exists a continuous splitting of the tangent bundle $TM = E^s \oplus E^c \oplus E^u$ and a Riemannian metric on $M$ such that for every $x\in M$, for every unit vector $v^{\ast}\in E^{\ast}(x), \ast=s,c,u$ it holds
 \[
 \|Df_x(v^s)\|<\min\{1,\|Df_x(v^c)\|\}, \max\{\|Df_x(v^c)\|,1\}<\|Df_x(v^u)\|.
 \]
\end{definition}

\smallskip

The bundles $E^s, E^c$ and $E^u$ are referred to as the stable, center and unstable bundles for $f$ respectively. Partial hyperbolic is a $C^1$-open condition, that is, there is an open neighborhood $C^1$ of $f$ where any $g$ partially hyperbolic diffeomorphism in that neighborhood is also partially hyperbolic. For more details and a introduction to Partial Hyperbolicity see \cite{pesin2004lectures}. 

Let us recall also the concept of foliation.

\smallskip

\begin{definition}
Let $M$ be a manifold of dimension $m$ and $\mathcal{W}=\{W(x)\}_{x\in M}$ a partition into $\mathcal{C}^r$ submanifolds of the same dimension $d$. We say that $\mathcal{F}$ is a foliation if there exists an open covering $\mathcal{U}=\{U\}$ of $M$, and for each $U\in\mathcal{U}$ a continuous function $\phi_U: (-1,1)^{m-d}\to \mathrm{Emb}^r\left((-1,1)^d,M\right)$ satisfying the following
\begin{enumerate}
\item If $x\in U$ then  there exists  unique $v\in (-1,1)^{m-d}$ and $h\in (-1,1)^d$ such that $\phi_U(v)(h)=x$. Moreover, the image of $\phi_U(v)$ is the connected component of $W(x)\cap U$ containing $x$.
\item If $U\cap U'$  then the change of coordinates map $\phi_{U',U}: (-1,1)^{m-d}\to (-1,1)^{m-d}$, $\phi_U(v)=v'$ if and only if $\phi_U(v)(h)=\phi_U(v)(h')$ is continuous.
\end{enumerate}
Moreover, if the change of coordinates above is differentiable we say that $\mathcal{W}$ is a differentiable foliation. The atoms $W(x)$ are the leaves of the foliation $\mathcal{W}$.  
\end{definition}

\begin{figure}[h]
    \centering
    \includegraphics[width=0.7\linewidth]{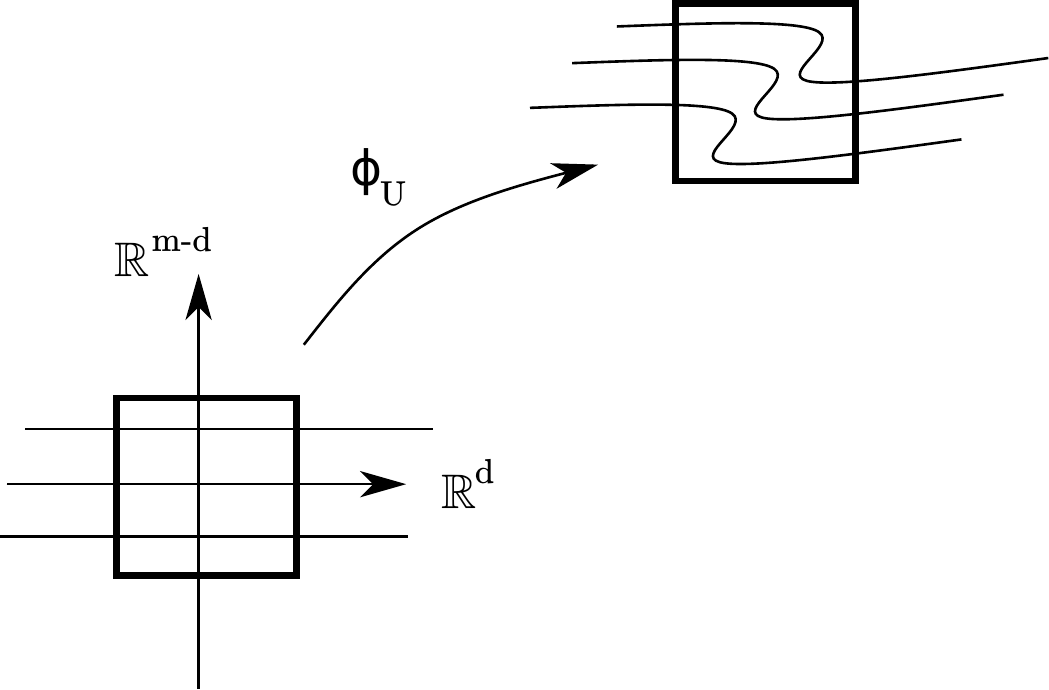}
     \label{fig:foliation}
\end{figure}

\begin{remark}
To be precise, our definition is what in the literature is known as a $\mathcal{C}^{0+,r}$ foliation. The reader can consult \cite{foliationsI} for more information and a introduction to the topic. 
\end{remark}

\smallskip

If $\mathcal{W}$ is a foliation then $T\mathcal{W}=\bigsqcup_{x} TW(x)$ is a sub-bundle of $TM$ that is (by definition) tangent to the leaves. Conversely:

\begin{definition}
A sub-bundle $E\subset TM$ is said to be integrable if there exists a foliation $\mathcal{W}$ on $M$ such that $T\mathcal{W}=E$.
\end{definition}

We now return to the Partially Hyperbolic context. By the classical Stable Manifold Theorem of Hadamard-Perron (cf. \cite{HPS}) both $E^s,E^u$ are integrable to $f$ - invariant foliations $\mathcal{W}^s, \mathcal{W}^u$; here, invariance means that $f$ permutes the leaves of $\mathcal{W}^{\ast}$ for $\ast=s,u$. In addition, if $d_{W^{\ast}(x)}$ denotes the intrinsic distance induced by the Riemannian metric in the leaf $W^{\ast}(x)$, then it is direct to verify that for some $\lambda<1$ it holds
\begin{align}
&\label{eq.contra1} y,y'\in W^s(x)\Rightarrow  d_{W^s(f^nx)}(f^ny,f^ny')\leq \lambda^n d_{W^s(x)}(y,y')\quad \forall n\geq 0\\ 
&\label{eq.contra2} z,z'\in W^u(x)\Rightarrow  d_{W^u(f^{-n}x)}(f^{-n}z,f^{-n}y')\leq \lambda^n d_{W^s(x)}(z,z')\quad \forall n\geq 0.
\end{align} 

\smallskip

On the other hand, the bundle $E^c$ is not always integrable (not even if $\dim E^c=1$, \cite{nondyncoh}). We say that a partially hyperbolic diffeomorphism $f$ is \emph{dynamically coherent} if the sub-bundles $E^{cs}:= E^s \oplus E^c$ and $E^{cu}:= E^c \oplus E^u$ integrable to $f$ - invariant foliations $\mathcal{W}^{cs}$ and $\mathcal{W}^{cu}$.  If $f$ is dynamically coherent, then $E^c$ is tangent to the foliation $\mathcal{W}^c := \mathcal{W}^{cs} \cap \mathcal{W}^{cu}$ obtained intersecting the leaves of $\mathcal{W}^{cs}$ and $\mathcal{W}^{cu}$.

Note the the time-one map of the geodesic flow referred before is dynamically coherent, with one dimensional center tangent to the flow lines. Not only that, its center foliation is in fact differentiable. This is seldom the case; in fact if a small volume preserving perturbation of $f$ has differentiable center, then in particular it is the time-one map of a conservative (hyperbolic) flow, which imposes a very strong constrain on the existence of such systems. See \cite{Avila2015}.

Nonetheless, small $\mathcal{C}^1$ perturbations of the geodesic flow are dynamically coherent. See the proof of Theorem B below.

\smallskip

A partially hyperbolic diffeomorphism $f$ is said to be \textit{accessible} if every pair $x,y \in M$ can be connected by a $\su$ path, that is, a path composed of segments that always lie in $\mathcal{W}^s$ or $\mathcal{W}^u.$ We say $f$ is $\mathcal{C}^r$ - \textit{stably accessible} if every $g$ sufficiently $C^r$ - close to $f$ is accessible. It was originally proved by Katok and Kononenko \cite{Katok1996} that the time-one map of a geodesic flow corresponding to a hyperbolic manifold is $\mathcal{C}^r$ - stably accessible.

\smallskip

We end this section by noting that if $\mathcal{C}$ is an $\su$ loop corresponding to a partially hyperbolic map $f$ (cf. Definition \ref{def:suloop}) then the functionals $F(\mathcal{C};x_i\to x_{i+1})$ appearing in the definition o{}f $F(\mathcal{C})$ are well defined on the space of H\"older continuous functions, in virtue of \eqref{eq.contra1} \eqref{eq.contra2}.

\section{Proof of the Main Result}\label{sec:proof}

In this part we put everything together and establish a result that will imply the Main Theorem. Let us fix $M$ a compact manifold and let us also fix $\mu$ a smooth volume on $M$. The set $\mbox{Diff}^r(M)$ consists of the $\mathcal{C}^r$ diffeomorphisms of $M$.

\smallskip

\begin{thmB}\label{teo.B}
Suppose that $f:M\to M$ is a $\mathcal{C}^{\infty}$ accessible partially hyperbolic diffeomorphism satisfying
\begin{enumerate}
\item $\dim E^c=1$.
\item For some Riemannian metric $Df|E^c$ is an isometry.
\item $f$ preserves the measure $\mu$.
\end{enumerate}
Let $p:M\to (0,1)$ be a H\"older continuous function, $q=1-p$, $\varphi=\frac{p}{q}$ and assume that $\int \log \varphi \cdot d\mu=0$.
Then there exists $N$ a $\mathcal{C}^2$ open neighborhood of $f$ such that if $g\in N, g_{\ast}\mu=\mu$ then the Random Walk on $M$ defined by $(g,p)$ has a $P$ - stationary measure equivalent to $\mu$ if and only if for every $\su$ loop the functional associated to $\mathcal{C}$ vanishes on $\log\varphi$. 

Moreover, the density of the $P$ - stationary measure is continuous. If furthermore $p$ is differentiable, then the $P$ - stationary measure is a smooth volume on $M$.
\end{thmB}

\begin{proof}
We start by noting that $f$ is $\mathcal{C}^2$ - stably-ergodic; meaning that there exists a $\mathcal{C}^2$ neighborhood $N$ of $f$ such that every $g\in N$ preserving $\mu$ is (partially hyperbolic and) ergodic (in particular, $\mu$ is an ergodic measure for $f$). One way to obtain this is as follows: due to Corollary $7.6$ in \cite{PartSurv}, the map $f$ is dynamically coherent, and used in conjunction with Theorem $7.4$ in \cite{HPS}, we obtain the same is also valid for $\mathcal{C}^1$ small perturbations. It is proven in \cite{hertz2008accessibility} that $f$ is also $\mathcal{C}^r$ stable accessible, for any $r\geq 2$. The stably ergodicity is then consequence of Theorem A in \cite{StableJulienne} (we remark that the technical condition of \emph{center bunching} in the hypotheses of that Theorem are immediate for perturbations of partially hyperbolic maps that act as an isometry on their centers).

It follows that there exists $N$ a $\mathcal{C}^2$ neighborhood of $f$ such that if $g\in N$ and $g_{\ast}\mu=\mu$ then
\begin{itemize}
\item $\mu$ is an ergodic measure for $g$.
\item $g$ is accessible.
\end{itemize}

Let us fix $g\in N$ preserving $\mu$. On the one hand, by Theorem \ref{teo.existestationary}, there exists a $P$ - stationary measure $\nu$ equivalent to $\mu$ if and only if there exists a integrable solution of the cohomological equation 
\[
    \log \varphi=\phi \circ g - \phi.
\]
On the other hand, since $\varphi$ is H\"older (and positive away from zero), then its logarithm is also H\"older and thus by Theorem A of \cite{wilkinson2013cohomological} (see also \cite{Katok1996}) we have:
\begin{enumerate}
\item the existence of a continuous solution to the previous equation is equivalent to the vanishing on $\log \varphi$ of any functional $F_g(\mathcal{C})$ associated to an $\su$ loop.
\item The existence of a measurable solution equation implies the existence of a continuous solution.
\end{enumerate}
From here follows the first part. 

Assuming further that $p$ is differentiable, we obtain by Corollary 0.2 of \cite{wilkinson2013cohomological} that solutions of the cohomological equation are automatically smooth, and since the density of the stationary measure is precisely $\phi$ (cf. remark after Theorem \ref{teo.existestationary}), we deduce that $\nu$ is a smooth volume. This finishes the proof. 
\end{proof}

\smallskip

We bring to the attention of the reader that if $f:M\to M$ is the geodesic flow corresponding to a compact hyperbolic manifold, then it satisfies the hypotheses of the above theorem, and thus Theorem A follows directly from Theorem B.

\begin{remark}
The previous theorem relies on accessibility and stable ergodicity, and thus its first part can be extended to more general situations. However, the smoothness part uses Corollary 0.2 of \cite{wilkinson2013cohomological} which requires center dimension equal to $1$ plus isometric behavior. Since the smoothness of the stationary measure does not appear on previous works, the authors opted to present Theorem B in its present form and leave for the interested reader the (direct) generalization of the first part.
\end{remark}

\section{Dynamical consequences}

In this part we deduce some consequences for the dynamics of $T$ of the existence of an stationary measure equivalent to Lebesgue. The set of $P$ - stationary measures is a simplex, thus by standard methods, if $\nu$ is $P$ - stationary for $f$ one deduces the existence of a extremal $P$ - stationary measure;  this implies that $\mathbb{Q}_{\nu}$ is ergodic for $T$, and we assume this to be case in what follows. We also suppose that (as given by Theorem B), the density of $\nu$ with respect to $\mu$ is continuous.

\smallskip

As explained in \ref{subsec:abscont}, $f^{-1}_{\ast}\nu=h\nu$, and thus $P$ induces an operator $P^{\dagger}:L^2(\nu)\to L^2(\nu)$ characterized by $\int Pf\cdot g d\nu=\int f\cdot P^{\dagger}g d\nu$ for every $f,g\in L^2(\nu)$. Since we are assuming that $\nu$ is equivalent to $\mu$ and $\int \log \frac{p}{q}d\mu=0$, by Theorem $1$ of \cite{Conze2000} it follows necessarily that $h=\frac{p}{q\circ T}$, and with this is direct to check that $P^{\dagger}=P$ is self-adjoint.

Let us make the following simple remark: given $\phi \in L^r(M,\nu), r\geq 1$ the function $\tilde{\phi}=\phi\circ X_0$ is in $L^r(\Omega,\mathbb{Q}_{\nu})$ and 
\[
\int  \tilde{\phi}(\omega)d\mathbb{Q}_{\nu}(\omega)=\int \phi(x) d\nu(x).    
\]
In the skew-product representation referred before, we think $\tilde{\phi}$ as a function depending only on $M$ coordinates in $X$. We then have the following.

\begin{theorem}
Let $\phi: M\to \Real$ be a measurable function. Then for Lebesgue almost every $x\in M$ it holds:
\begin{enumerate}
\item if $\phi \in L^1(M,\mu)$ then 
\[
    \frac{1}{n}\sum_{k=0}^{n-1} \phi\left(\alpha_k\cdots \alpha_1(x)\right)\xrightarrow[n\to\infty]{}\int \phi d\nu   \quad \mP_x-a.e.(\alpha_k)_{k\geq 1}\in \Sigma.
\] 
\item If $\phi=\psi-P\psi$ (or, more generally, $\phi=(I-P)^{1/2}\psi$) for some $\psi \in L^2(M,\mu)$, then
\[
   \frac{1}{\sqrt{n}}\sum_{k=0}^{n-1} \phi\left(\alpha_k\cdots \alpha_1(x)\right)\xrightarrow[n\to\infty]{\mathrm{dist}}\mathcal{N}(0,\sigma^2)   \quad \mP_x-a.e.(\alpha_k)_{k\geq 1}\in \Sigma, 
\]
where $\mathcal{N}(0,\sigma^2)$ denotes the normal standard distribution centered at zero, with variance $\sigma^2=\|\psi\|_{L^2(\nu)}^2-\|P\psi\|_{L^2(\nu)}^2$.
\end{enumerate}
\end{theorem}

\begin{proof}
Note that since the density of $\nu$ with respect to $\mu$ is continuous, in particular any $\phi\in L^r(\mu)$ represents a function in $L^r(\nu)$ which we denote by the same letter; thus with no loss of generality we consider $\phi \in L^r(\nu)$.
The first part is direct consequence of Birkhoff's ergodic theorem applied to the dynamical system $(T,\mathbb{Q}_{\nu})$ together with Proposition \ref{pro.desintegra}. The second part follows from the same referred Proposition, and a by now classical result of Gordin and Lifsic \cite{GorLif} (respectively, Kipnis and Varadhan \cite{KipVar} when $\phi=(I-P)^{1/2}\psi$; here we use that $P$ in $L^2(\nu)$ is self adjoint).
\end{proof}

\smallskip

Given the skew-product $F(\alpha,x)=(\sigma(\alpha),\alpha_1(x))$ defined in (\ref{eq.skew}), for each fixed $\alpha \in \Sigma$, we define $F_{\alpha}: M \to M$ by the projection on the second coordinate of $F,$ that is, $F_\alpha(x) = \pi_2(F(\alpha,x)).$ Thus $F^n_\alpha(x) = \alpha_n \dots \alpha_1(x)$ for all $n\in \mathbb{N}.$  We say that $x$ is recurrent for $F_\alpha$ if $x \in \omega_{F_\alpha}(x):= \{y\in M: \exists (n_k)_k \to \infty \mbox{ such that } \lim_{k\to \infty} F_\alpha^{n_k}(x)=y\}.$ 

Let $\mathcal{O}_f(x)$ denote the orbit of $x$, and by $\mathcal{O}_{F_\alpha}(x)$ the orbit of $x$ by projection on the second coordinate the of skew-product $F$.  

\smallskip

\begin{proposition}\label{recorrenciafibra}
Let $f$ time-one map of the geodesic flow whose Markov chain determined by $(f,p)$ admits a stationary measure $\nu$ equivalent to $\mu.$ Then for Lebesgue almost every $x\in M$ we have that: 
\begin{enumerate}
    \item $x$ is recurrent for $F_\alpha$, for $\mathbb{P}_x$-almost every $\alpha\in \Sigma.$
    \item The orbit $\mathcal{O}_{F_\alpha}(x)$ is dense in $M$ for $\mathbb{P}_x$-almost every $\alpha \in \Sigma.$ 
\end{enumerate}
 \end{proposition}

\begin{proof}
The first part follows for example by Theorem 2.1 in \cite{RWRE}, which implies that in the symmetric case, for $\nu$ almost every $x$ (and hence for $\mu$ almost every $x$) it holds that $\lim \inf_{n\to \infty}X_n = -\infty$ and $\lim \sup_{n\to \infty}X_n = \infty$ with full $\mathbb{P}_x$ probability. Note that for $\mathbb{P}_x-$almost every $\alpha \in \Sigma$ there are an infinity of $n_k \in \mathbb{N},$ such that $\alpha_{n_k}\dots\alpha_1(x)= x.$ Then $\lim_{n_k\to \infty}F_{\alpha}^{n_k}(x) = x.$ 

For the second part,  observe that since $f$ is conservative and topologically transitive, it holds  $\mu\big(\{x\in M: \mathcal{O}_f(x) \mbox{ is dense in } M\}\big)=1$ and, by the previous part $\mu\big(\{x\in M: x\text{ is recurrent for }F_\alpha\, \, \text{for} \; \mathbb{P}_x$ - almost every $\alpha\in \Sigma.\}\big)=1.$ It follows that for $\nu$ - almost every $x\in M$ the orbit $\mathcal{O}_{F_\alpha}(x)$ is dense in $M$ for $\mathbb{P}_x$ - almost every $\alpha\in \Sigma$.   
\end{proof}

\smallskip

To simplify the presentation for the last part we assume now that $S$ is a two-dimensional compact hyperbolic surface, and $f:M=T^1S\to M$ is the time-one map of the geodesic flow; in particular $\dim E^{\ast}=1 \forall\ast=s,c,u$. For $(\alpha,x)\in X$ define $A^{u}(\alpha,x)=A^{u}(\alpha_1,x)$ by
\[
A^u(\alpha,x):= \|Df^{\alpha_1}_x|E^{u}_x\|. 
\]
Then Birkhoff's ergodic theorem guarantees the existence of $m_{C}$ - almost every $(\alpha,x)$ of the Lyapunov exponent
\begin{align*}
\lim_{n\to\infty}\frac{1}{n}\sum_{k=0}^{n-1} \log A^{u}(F^k(\alpha,x))=\int \log A^{u} dm_{C}.
\end{align*}
Using \eqref{eq.PHI} and Theorem \ref{teo.canonicalchain} we obtain,
\begin{align*}
 &\int \log A^{u} dm_{C}=\int \log A^{u}\Phi^{-1}(\omega)d\mathbb{Q}_{\nu}(w)\\
 &=\int \log \|Df_x|E^u_x\|p(x)d\nu(x)+ \int \log \|Df^{-1}_x|E^u_x\|q(x)d\nu(x).  
\end{align*}
On the other hand, the existence of the limit together with the recurrence given in Proposition \ref{recorrenciafibra} implies that this limit has to be zero, which in turn leads to
\begin{equation}
\int \log \|Df_x|E^u_x\|p(x)d\nu(x)=- \int \log \|Df^{-1}_x|E^u_x\|q(x)d\nu(x).
\end{equation}
If $N$ is an algebraic surface, then the derivative of $f$ has constant exponents and  the above implies (quite indirectly) that
\[
    \int p(x)d\nu(x)=\int q(x)d\nu(x).
\]

\noindent\textbf{Question:} If $\nu$ is the $P$ - stationary measure equivalent to $\mu$, does necessarily the previous equality hold?

\medskip

\section*{Acknowledgments}

The results in this article are part of the PhD dissertation of the second author, who was co-advised by the first author and S\^onia Pinto-de-Carvalho. Both authors thank S\^onia for the interest deposited in this project. We learned about random walks defined by dynamical systems in a mini-course taught to B. Fayad in the V Brazilian School in Dynamical Systems; we thank Fayad for introducing us to this type of mathematics. 

Finally, we thank for the reviewer for the suggestions that improved our manuscript.

\section*{Appendix: Anosov Case}

Here we indicate how the discussion in this article is also valid when $f$ is a ($\mathcal{C}^2$. conservative) Anosov diffeomorphism. Since this case was already studied in the literature (cf. \cite{Conze2000}, \cite{simpleanosov}) we limit to point out succinctly the necessary modifications.

\smallskip

\begin{theorem}
Let $f:M\to M$ be a conservative Anosov diffeomorphism of class $\mathcal{C}^2$ and let $p:M\to (0,1)$ be a H\"older continuous function, $q=1-p, \varphi=\frac{p}{q}$. Then there exists a $\mathcal{C}^2$ neighborhood $N$ of $f$ such that if $g\in N$ is conservative then the Random Walk defined $(g,p)$ has a $P$ - stationary measure equivalent to the volume of $M$,  if and only if for every $g$ - invariant probability measure $\eta$ it holds $\int \log \varphi\cdot d\eta =0$. 

Moreover, the density of the $P$ - stationary measure is continuous. If furthermore $p$ is differentiable, then the $P$ - stationary measure is a smooth volume on $M$.
\end{theorem}

\smallskip

The proof of this result follows exactly the same lines of Theorem B by using Livshitz' theory for hyperbolic maps instead of \cite{wilkinson2013cohomological}. Namely, if $f:M\to M$ is Anosov and $\varphi:M\to\Real_{>0}$ is H\"older, then there exists solution of the cohomological equation $\log\varphi=\phi\circ f-\phi$ if and only if $\int \log\varphi d\eta=0$
for every $f$-invariant measure $\eta$. The solution of the equation also has rigidity properties as in Theorem B, i.e.\@ if $\varphi$ is differentiable then the solution is differentiable, provided that it exists. See \cite{HK} for a discussion of these (classical) results. 

We remark that $\mathcal{C}^2$ conservative Anosov diffeomorphisms are ($\mathcal{C}^2$ - stably) ergodic, as proven by Anosov \cite{AnosovThesis}.

\bibliographystyle{siam}
\bibliography{bib}
\end{document}